\documentclass{amsart}

\newtheorem{theorem}{Theorem}[section]
\newtheorem{lemma}[theorem]{Lemma}

\theoremstyle{definition}
\newtheorem{definition}[theorem]{Definition}

\theoremstyle{remark}

\numberwithin{equation}{section}

\begin{document}

\title[ Ricci $\rho$-Solitons on 3-dimensional $\eta$-Einstein almost
Kenmotsu manifolds ]{Ricci $\rho$-Solitons on 3-dimensional $\eta$-Einstein almost
Kenmotsu manifolds }

\author{Shahroud Azami and Ghodratallah Fasihi-Ramandi}
\address{Department of Mathematics, Faculty of Sciences, Imam Khomeini International University, Qazvin, Iran. }

\email{azami@sci.ikiu.ac.ir and fasihi@sci.ikiu.ac.ir}



%    General info
\subjclass[2010]{53Axx, 53Bxx}

%\date{April  6, 2013 .}


\keywords{Almost Kenmotsu manifold, $\rho$-Ricci soliton, $\eta$-Einstein, Generalized k-nullity distribution.}
%%%%%%%%%%%%%%%%%%%%

\begin{abstract}
The notion of quasi-Einstein metric in theoretical physics and in relation with string theory is equivalent to the notion of Ricci soliton in differential geometry. Quasi-Einstein metrics or Ricci solitons serve also as solution to Ricci flow equation, which is an evolution equation for Riemannian metrics on a Riemannian manifold. Quasi-Einstein metrics are subject of great interest in both mathematics and theoretical physics.
In this paper the notion of Ricci $\rho$-soliton as a generalization of Ricci soliton is defined. We are motivated by the Ricci-Bourguignon flow to define this concept. We show that if a 3-dimensional almost Kenmotsu  Einstein manifold M be a $\rho$-soliton, then M is a Kenmotsu manifold of constant sectional curvature $-1$ and the $\rho$-soliton is expanding, with $\lambda=2$.
\end{abstract}

\maketitle
%%%%%%%%%%%%%%%%%%%%%%%%%%%
%%%%%%%%%%%%%%%%%%%%%%%%%%%%%%%%%%%%%
\section{Introduction}
Ricci flow and other geometric flows are an active subject of current research in physics and mathematics. The notion of Ricci-Bourguignon flow as a generalization of Ricci flow has been introduced in \cite{5}. The Ricci-Bourguignon flow is an evolutionary equation for Riemannian metrics on a manifold $M^n$ as follows.
\begin{equation}
\dfrac{\partial g}{\partial t}=-2 (\mathrm{Ric} -\rho Rg), \quad g(0)=g_0
\end{equation}
where, $\mathrm{Ric}$ is the Ricci curvature tensor, $R$ is the scalar curvature with respect to $g$ and $\rho$ is a real non-zero constant. Short time existence and uniqueness for the solution of this geometric flow has been proved in \cite{6}. In fact, for sufficiently small $t$ the equation has a unique solution for $\rho < {1/2(n-1)}$.\\
In the other hand, quasi Einstein metrics or Ricci solitons serve as a solution to Ricci flow equation. This motivates a more general type of Ricci soliton by considering the Ricci-Bourguignon flow. In fact, a Riemannian manifold $(M, g)$ of dimension $n\geq 3$ is said to be Ricci $\rho$-soliton if
\begin{equation}\label{soliton}
\dfrac{1}{2}\mathcal{L}_V g +\mathrm{Ric} +(\lambda +\rho R)g=0,
\end{equation}
where, $\mathcal{L}_V$ denotes the Lie derivative operator along vector field $V$ and $\lambda$ is an arbitrary real constant. Similar to Ricci solitons, a Ricci $\rho$ soliton is called expanding if $\lambda >0$, steady if $\lambda =0$ and
shrinking if $\lambda <0$. If the vector field $V$ is the gradient of a smooth function $f\in C^\infty (M)$, then $(M,g)$ is called a gradient $\rho$-soliton. Hence, (\ref{soliton}) reduces to the form
\begin{equation}
\mathrm{Hess}f +\mathrm{Ric} +(\lambda +\rho R)g=0.
\end{equation}

Recently, Ricci solitons and gradient Ricci solitons on some kinds of three dimensional almost contact metric manifolds have been studied by many authors. For instances, Ricci solitons and gradient Ricci solitons on
three-dimensional normal almost contact metric manifolds are investigated in \cite{8} and three-dimensional trans-Sasakian manifolds are considered in \cite{14}. Moreover, a complete classification of
Ricci solitons on three-dimensional Kenmotsu manifolds is given (see \cite{10} and \cite{7}). Also, in \cite{15} Wang and Liuva showed that if the metric $g$ of a three-dimensional $\eta$-Einstein almost Kenmotsu manifold $M$ be a Ricci soliton, then $M$ is a Kenmotsu manifold of constant sectional curvature $-1$ and the soliton is expanding. Generalizing some corresponding results of
the paper \cite{15}, the present paper is devoted to investigating Ricci $\rho$-solitons on a type of almost Kenmotsu manifolds
of dimension three, namely, $\eta$-Einstein almost Kenmotsu manifolds.\\
This paper is organized as follows. In the preliminaries section, we recall some well known basic formulas and properties of almost Kenmotsu manifolds. In section 3, we completely classify Ricci $\rho$-solitons on a three dimensional  almost Kenmotsu manifold such that the Reeb vector field belongs to the generalized $k$-nullity distribution. Moreover, an example of such manifolds can also be seen in the last section.
%%%%%%%%%%%%%%%%%%%%%%%%%%%%%%%%%%%%%
%%%%%%%%%%%%%%%%%%%%%%
\section{Preliminaries}
In this section we summarize some basic definitions on contact manifolds, with emphasis on those
aspects that will be needed in the next section. For more details one can consult \cite{4}.

\begin{definition}
An almost contact structure on a $(2n+1)$-dimensional smooth manifold $M$ is a triple $(\phi ,\xi , \eta)$, where $\phi$ is a $(1, 1)$-type tensor field, $\xi$ is a global vector field and $\eta$ a 1-form, such that
\begin{equation}
\phi^2 =-\mathrm{id}+\eta \otimes \xi, \quad \eta (\xi)=1,
\end{equation}
where, $\mathrm{id}$ denotes the identity mapping, which imply that $\phi (\xi)=0$, $\eta \circ \phi=0$ and
$\mathrm{rank}(\phi)=2n$. Generally, $\xi$ is called the characteristic vector field or the Reeb vector field.
\end{definition}
As mentioned, contact manifolds are endowed with extra structures rather than differential structure, so it is natural to consider special metrics on these manifold in which some conditions of compatibility are requested for them.
\begin{definition}
A Riemannian metric $g$ on $M^{2n+1}$ is said to be compatible with the almost contact structure $(\phi ,\xi ,\eta )$ if for every $X,Y\in \mathcal{X}(M)$, we have
\[
g(\phi (X),\phi (Y)) = g(X,Y)) - \eta (X)\eta (Y).
\]
\end{definition}
An almost contact structure endowed with a compatible Riemannian metric is said to be an almost contact metric structure. Also, the
fundamental $2$-form Φ of an almost contact metric manifold $M^{2n+1}$ is defined by 
\[
\Phi (X,Y) = g(X,\phi (Y))
\]
for any vector fields $X$, $Y$ on $M^{2n+1}$.
\begin{definition}
If $(M,\phi ,\xi, \eta ,g)$ be an almost contact metric structure, then there is a well known deformation of contact forms which is named D-homothetic deformation and is defined by
\[\bar{\eta}=a\eta , \quad \bar{\phi}=\phi,\quad \bar{\xi}=\dfrac{1}{a}\xi ,\quad \bar{g}=ag+a(a-1)\eta\otimes \eta,
\]
where, $a$ is a positive constant.
\end{definition}
Also, we have the following definitions and concepts in contact manifolds.
\begin{definition}
An almost Kenmotsu manifold is defined as an almost contact metric manifold such that $d\eta =0$ and $d\Phi =2\eta \wedge \Phi$. Also, An almost Kenmotsu manifold is said to be $\beta$-Kenmotsu manifold if for all vector field $X$ and $Y$ on $M$, we have
\[(\nabla_X \phi )Y=\beta [g(\phi(X),Y)\xi-\eta (Y)\phi (X)]\]
where, $\nabla$ is the Levi-Civita connection with respect to $g$ and $\beta$ is a smooth funcyion on $M$. If $\beta=1$ definition of Kenmotsu manifold is obtained.
\end{definition}
Local structure of Kenmotsu manifolds is determined in \cite{11}.
\begin{theorem}{\cite{11}}
A Kenmotsu manifold $M^{2n+1}$ is locally isometric to a warped product $I\times_\theta M^{2n}$, where $M^{2n}$ is a Kahlerian manifold,
$I$ is an open interval with coordinate $t$ and the warping function $\theta =ce^t$ for some positive constant $c$.
\end{theorem}
\begin{definition}
On an almost contact metric manifold $M$, if the Ricci operator satisfies
\begin{equation}\label{einstein}
\mathrm{Ric} = \alpha g + \beta \eta  \otimes \eta 
\end{equation}
where $\mathrm{Ric}$ is the Ricci curvature tensor and both $\alpha$ and $\beta$ are smooth functions on $M$, then $M$ is said to be an $\eta$-Einstein manifold. 
\end{definition}
Obviously, an $\eta$-Einstein manifold with vanishing $\beta$ and $\alpha$ a constant is an
Einstein manifold. An $\eta$-Einstein manifold is said to be proper $\eta$-Einstein if $\beta \neq 0$.\\

Finally, remind that there are two natural tensor fields (with respect to metric contact structure) on an almost metric contact manifold. Set
\begin{equation}
h=\dfrac{1}{2}\mathcal{L}_\xi \phi ,\qquad \ell =R(.,\xi)(\xi),
\end{equation}
where, $R$ denotes the Riemannian curvature tensor related to $g$. One can easily check that both $h$ and $\ell$ are symmetric tensor fields and satisfy the following equations.
\[
\ell (\xi ) = 0,\,\,\,\,\,h(\xi ) = 0,\,\,\,\,\,\,\mathrm{tr}(h) = 0,\,\,\,\,\,\mathrm{tr}(h\phi ) = 0,\,\,\,\,\,h \circ \phi  + \phi  \circ h = 0.
\]
Also, the following identities are proven in \cite{4}.
\begin{align}
\nabla _X \xi  &= \phi ^2 (\xi ) + h'X, \\
\phi \ell \phi  - \ell  &= 2(h^2  - \phi ^2 ),\\
\mathrm{tr}(\ell ) &= \mathrm{Ric}(\xi ,\xi ) =g(\mathrm{Rc}(\xi),\xi)=  - 2n - \mathrm{tr}(h^2 ),\\
R(X,Y)\xi  &= \eta (X)(Y - \phi hY) - \eta (Y)(X - \phi hX) + (\nabla _Y \phi h)X - (\nabla _X \phi h)Y,
\end{align}
where, $h'=h\circ \phi$ and $\mathrm{Rc}$ is Ricci operator with respect to $g$.
%%%%%%%%%%%%%%%%%%%%%%%%%%%%%%%%%%%%%%%%
\section{Main Results}
In this section $(M,g)$ is a three-dimensional almost Kenmotsu manifold. If the characteristic vector field $\xi$ of $M$ belongs to generalized $k$-nullity distribution defined by
\[
R(X,Y)\xi  = k[\eta (Y)X - \eta (X)Y]
\]
then Proposition 3.1 of \cite{13} guarantees that $M$ is a $\eta$-Einstein manifold and vice versa. Moreover, the function $k$ in the above formula can be expressed by $k=(\alpha +\beta)/2$. \\
Let $(M, g)$ be an almost Kenmotsu manifold of dimension 3 with $\xi$ belonging to the generalized$k$-nullity distribution. The following formulas are proven in \cite{12}.
\begin{align}
&h^2  = h'^2  = (k + 1)\phi ^2, \label{eq}\\
&\qquad \mathrm{Rc}(\xi ) = 2k\xi \nonumber .
\end{align}
Then the above equation follows that $k\leq -1$ everywhere on $M$. Moreover, $k=-1$ holds if and
only if $h=h'=0$. If $k<-1$, we denote the two non-zero
eigenvalues of $h$ by $\nu$ and $-\nu$ respectively, where $\nu=\sqrt{-1-k}>0$.
Furthermore, by Proposition 3.1 of \cite{8} we also have
\[\nabla_\xi h'=-2h'.\]
We need the following results from \cite{15} for proving our main theorem.
\begin{lemma}\cite{15}\label{lem}
Let $(M, g)$ is an almost Kenmotsu manifold of dimension 3 such that the Reeb vector field belongs to the generalized $k$-nullity distribution, then we have
\[\vec{\nabla} k=-4(k+1)\xi\]
where $\vec{\nabla} $ denotes the gradient operator with respect to $g$.
\end{lemma}
\begin{lemma}\cite{15}
Let $(M,g)$ be a three-dimensional almost Kenmotsu manifold such that the characteristic vector field belongs to the generalized $k$-nullity distribution, then either $k=-1$ identically or $k<-1$ everywhere on $M$.
\end{lemma}
Now, we are ready to present our main theorem.
\begin{theorem}
Let the metric $g$ of a three-dimensional $\eta$-Einstein almost Kenmotsu
manifold $(M,g)$ be a Ricci $\rho$-soliton, then $M$ is a Kenmotsu manifold of constant sectional curvature $-1$ and the soliton is expanding with $\lambda=2$.
\end{theorem}
\begin{proof}
According to previous lemma, we prove the theorem in two cases where $k=-1$ identically and $k<-1$ everywhere on $M$.\\
{\bf Case 1:} Suppose that we have $k<-1$ everywhere on $M$ which is equivalent to $h\neq 0$. Putting relation (\ref{einstein}) into (\ref{soliton}) we obtain
\begin{equation}\label{43}
\mathcal{L}_V g =  - 2(\alpha  + \rho R + \lambda )g - 2\beta \eta  \otimes \eta.
\end{equation}
Taking the covariant differentiation from both sides of the above formula along an arbitrary vector field $X$ we obtain the following equality for any vector fields $Y$ and $Z$ on $M$.

\begin{align}\label{35}
  (\nabla _X \mathcal{L}_V g)(Y,Z) =&  - 2(X(\alpha ) + \rho X(R))g(Y,Z) - 2X(\beta )\eta (Y)\eta (Z) \nonumber \\
   &-2\beta g(X + h'X,Y)\eta (Z) - 2\beta g(X + h'X,Z)\eta (Y) \\
   &+ 2\beta \eta (X)\eta (Y)\eta (Z). \nonumber
\end{align}
But we know the following formula from Yano \cite{16},
\[
(\mathcal{L}_V \nabla _X g - \nabla _X \mathcal{L}_V g - \nabla _{[V,X]} g)(Y,Z) =  - g((\mathcal{L}_V \nabla )(X,Y),Z) - g((\mathcal{L}_V \nabla )(X,Z),Y).
\]
Since $\nabla$ is the Levi-Civita connection of $M$ we have $\nabla g=0$ and then the above formula becomes
\[
(\nabla _X \mathcal{L}_V g)(Y,Z) = g((\mathcal{L}_V \nabla )(X,Y),Z) + g((\mathcal{L}_V \nabla )(X,Z),Y).
\]
One can easily check that the operator $(\mathcal{L}_V \nabla )$ is a symmetric tensor field of type $(1,2)$ i.e., $(\mathcal{L}_V \nabla )(X,Y)=(\mathcal{L}_V \nabla )(Y ,X)$. In fact, this symmetry is a consequence of Jacobi identity in the Lie algebra of smooth real function on $M$. Hence, a simple combinatorial argument shows that
\begin{equation}\label{73}
g((\mathcal{L}_V \nabla )(X,Y),Z) = \frac{1}
{2}(\nabla _X \mathcal{L}_V g)(Y,Z) + \frac{1}
{2}(\nabla _Y \mathcal{L}_V g)(Z,X) - \frac{1}
{2}(\nabla _Z \mathcal{L}_V g)(X,Y).\,\,
\end{equation}
Using (\ref{73}) and (\ref{35}) the following formula is obtained,
\begin{align}
  (\mathcal{L}_V \nabla )(X,Y) =&  - (X(\alpha ) + \rho X(R))Y - (Y(\alpha ) + \rho Y(R))X + g(X,Y)\vec{\nabla} \alpha \label{83}\\
    &+ \rho g(X,Y)\vec{\nabla} R + \eta (X)\eta (Y)\vec{\nabla} \beta  - [X(\beta )\eta (Y) + 2\beta g(X + h'X,Y) \nonumber\\
  &- 2\beta \eta (X)\eta (Y) + Y(\beta )\eta (Y)]\xi .  \nonumber 
\end{align}
Considering an orthonormal local frame $\{e_i\}_{i=1}^3$ on $M$ and replacing  $X$ and $Y$ by $e_i$ and summing over $i=1,2,3$, we have
\begin{equation}\label{93}
\sum\limits_{i = 1}^3 {(\mathcal{L}_V \nabla )(e_i ,e_i )}  = \vec{\nabla} \alpha  + \vec{\nabla} \beta  + \rho \vec{\nabla} R - 2[\xi (\beta ) + 2\beta ]\xi .
\end{equation}
On the other hand, taking the covariant differentiation of the Ricci soliton equation (\ref{soliton}) along an arbitrary vector field $X$ we obtain $\nabla_X \mathcal{L}_V g=-2\rho X(R)g-2\nabla_X \mathrm{Ric}$, putting this relation into (\ref{73}) we obtain
\[
\begin{gathered}
  g((\mathcal{L}_V \nabla )(X,Y),Z) = (\nabla _Z \mathrm{Ric})(X,Y) - (\nabla _X \mathrm{Ric})(Y,Z) - (\nabla _Y \mathrm{Ric})(X,Z)\,\,\,\,\,\,\,\,\,\,\,\,\,\,\,\,\,\,\,\,\,\,\,\,\,\,\,\,\,\,\, \hfill \\
  \,\,\,\,\,\,\,\,\,\,\,\,\,\,\,\,\,\,\,\,\,\,\,\,\,\,\,\,\,\,\,\,\,\,\,\,\,\,\,\,\,\,\,\,\,\,\quad\quad + \rho Z(R)g(X,Y) - \rho X(R)g(Y,Z) - \rho Y(R)g(X,Z). \hfill \\
\end{gathered}
\]
Replacing $X=Y=e_i$ in the above formula and summing over $i=1,2,3$, we obtain $\sum\limits_{i=1}^3 (\mathcal{L}_V \nabla)(e_i,e_i)=\rho \vec{\nabla}R$, and this relation with (\ref{93}) gives us the following equation
\begin{equation}\label{113}
\vec{\nabla} \alpha  + \vec{\nabla} \beta  - 2[\xi (\beta ) + 2\beta ]\xi  = 0.
\end{equation}
Using the relation (\ref{83}) and taking the covariant differentiation of $(\mathcal{L}_V \nabla)(Y,Z)$ along an arbitrary vector field $X$, we may obtain
\begin{align}\label{132}
  &(\nabla _X \mathcal{L}_V \nabla )(Y,Z) \\
  =& - g(Y,\nabla _X \vec{\nabla} \alpha )Z - g(Z,\nabla _X \vec{\nabla} \alpha )Y - \rho g(Y,\nabla _X \vec{\nabla} R)Z \nonumber\\
  &- \rho g(Z,\nabla _X \vec{\nabla} R)Y + \eta (Y)\eta (Z)\nabla _X \vec{\nabla} \beta  + g(Y,Z)\nabla _X \vec{\nabla} \alpha  + \rho g(Y,Z)\nabla _X \vec{\nabla} R \nonumber \\
&+ [g(X + h'X,Z)\eta (Y) + g(X + h'X,Y)\eta (Z) - 2\eta (X)\eta (Y)\eta (Z)]\vec{\nabla} \beta \nonumber\\
 &- [Y(\beta )\eta (Z) + 2\beta g(Y + h'Y,Z) - 2\beta \eta (Y)\eta (Z) + \eta (Y)Z(\beta )](X + h'X) \nonumber\\
 &- g(Y,\nabla _X \vec{\nabla} \beta )\eta (Z)\xi  - g(Z,\nabla _X \vec{\nabla} \beta )\eta (Y)\xi  - 2\beta g((\nabla _X h')Y,Z)\xi \nonumber \\
&+ 2\beta [g(X + h'X,Y)\eta (Z) + g(X + h'X,Z)\eta (Y) + g(Y + h'Y,Z)\eta (X)]\xi \nonumber \\
&- Y(\beta )[g(X + h'X,Z) - 2\eta (X)\eta (Z)]\xi  - Z(\beta )[g(X + h'X,Y) - 2\eta (X)\eta (Y)]\xi \nonumber \\
&- X(\beta )[g(Y + h'Y,Z) - \eta (Y)\eta (Z)]\xi  - 6\eta (X)\eta (Y)\eta (Z)\xi .  \nonumber 
\end{align}
The following tonsorial identity is well known (see \cite{16}),
\begin{equation}\label{133}
(\mathcal{L}_V R)(X,Y)Z = (\nabla _X \mathcal{L}_V \nabla )(Y,Z) - (\nabla _Y \mathcal{L}_V \nabla )(X,Z),
\end{equation}
for any vector fields $X$, $Y$, and $Z$.\\
Also, note that for any smooth function $f$ on a Riemannian manifold $(M,g)$ we have $g(\nabla_X \vec{\nabla} f ,Y)=g(\nabla_Y \vec{\nabla} f ,X)$. Applying this fact and using the relations (\ref{133}) and (\ref{132}), by a straightforward computation we obtain
\begin{equation}\label{143}
\begin{gathered}
  (\mathcal{L}_V R)(X,Y)Z = g(Z,\nabla _Y \vec{\nabla} \alpha )X - g(Z,\nabla _X \vec{\nabla} \alpha )Y + \rho g(Z,\nabla _Y \vec{\nabla} R)X - \rho g(Z,\nabla _X \vec{\nabla} R)Y \hfill \\
\,\,\,\,\,\,\,\,\,\,\,\,\, + [g(X + h'X,Z)\eta (Y) - g(Y + h'Y,Z)\eta (X)]\vec{\nabla} \beta  + \eta (Z)[\eta (Y)\nabla _X \vec{\nabla} \beta  - \eta (X)\nabla _Y \vec{\nabla} \beta ] \hfill \\
\,\,\,\,\,\,\,\,\,\,\,\,\, + g(Y,Z)\nabla _X \vec{\nabla} \alpha  - g(X,Z)\nabla _Y \vec{\nabla} \alpha  + \rho g(Y,Z)\nabla _X \vec{\nabla} R - \rho g(X,Z)\nabla _Y \vec{\nabla} R\,\,\,\,\,\,\,\,\,\,\,\,\,\,\,\,\,\,\,\,\,\,\,\,\,\,\,\,\,\,\,\,\,\, \hfill \\
\,\,\,\,\,\,\,\,\,\,\,\,\,+ [X(\beta )\eta (Z) + 2\beta g(X + h'X,Z) - 2\beta \eta (X)\eta (Z) + \eta (X)Z(\beta )](Y + h'Y) \hfill \\
\,\,\,\,\,\,\,\,\,\,\,\,\,- [Y(\beta )\eta (Z) + 2\beta g(Y + h'Y,Z) - 2\beta \eta (Y)\eta (Z) + \eta (Y)Z(\beta )](X + h'X) \hfill \\
\,\,\,\,\,\,\,\,\,\,\,\,\, - X(\beta )g(Y + h'Y,Z)\xi  + Y(\beta )g(X + h'X,Z)\xi  - 2\beta g((\nabla _X h')Y,Z)\xi  \hfill \\
\,\,\,\,\,\,\,\,\,\,\,\,\,- [g(Z,\nabla _X \vec{\nabla} \beta )\eta (Y) - g(Z,\nabla _Y \vec{\nabla} \beta )\eta (X)]\xi  + 2\beta g((\nabla _Y h')X,Z)\xi,  \hfill \\
\end{gathered}
\end{equation}
for any vector fields $X$, $Y$, and $Z$.\\
Consider again the local orthonormal frame $\{e_i\}_{i=1}^3$, remind that for any smooth function $f$ on the Riemannian manifold $(M,g)$, the Laplace operator $\triangle$ acts on $f$ by 
\[
\triangle (f)=-\sum \limits_{i=1}^3 g(\nabla_{e_i}\vec{\nabla} f ,e_i). \]
Contracting the tonsorial relation (\ref{143}) over $X$, then a straightforward computation shows
\begin{align}\label{153}
  (\mathcal{L}_V \mathrm{Ric})(Y,Z) =&  - g(Y,Z)\triangle \alpha  - \rho g(Y,Z)\triangle R - \eta (Y)\eta (Z)\triangle \beta  + 2g(Z,\nabla _Y \vec{\nabla} \alpha )\\
   &+ 2\rho g(Z,\nabla _Y \vec{\nabla} R)- 2\xi (\beta )g(Y + h'Y,Z) + \eta (Y)g(Z + h'Z,\vec{\nabla} \beta )\nonumber \\
    &+ \eta (Z)g(Y + h'Y,\vec{\nabla} \beta )
  - \eta (Y)\eta (\nabla _Z \vec{\nabla} \beta ) - \eta (Z)\eta (\nabla _Y \vec{\nabla} \beta ) - 4\beta g(Y,Z)\nonumber \\
   &- 2\beta g(h'Y,Z)+ 4\beta \eta (Y)\eta (Z)- 2Z(\beta )\eta (Y) - 2Y(\beta )\eta (Z). \nonumber
\end{align}
Moreover, keeping in mind that $M$ is an $\eta$-Einstein manifold, by (\ref{einstein}) and a straightforward calculation we obtain that
\begin{equation}\label{163}
\begin{gathered}
  (\mathcal{L}_V \mathrm{Ric})(Y,Z) = [V(\alpha ) - 2\alpha (\lambda  + \alpha  + \rho R)]g(Y,Z) + [V(\beta ) - 2\alpha \beta  - 2\beta \eta (V)]\eta (Y)\eta (Z) \hfill \\
  \,\,\,\,\,\,\,\,\,\,\,\,\,\,\,\,\,\,\,\,\,\,\,\,\,\,\,\,\,\,\,\,\, + \beta [g(Z + h'Z,V) + \eta (\nabla _Z Y)]\eta (Y) + \beta [g(Y + h'Y,V) + \eta (\nabla _Y Z)]\eta (Z),\,\,\,\, \hfill \\
\end{gathered}
\end{equation}
for any vector fields $Y,Z\in \mathcal{X}(M)$.\\
Subtracting (\ref{153}) from (\ref{163}) gives an equation, substituting $Y$ and $Z$ with $\phi Y$
and $\phi Z$ respectively in the resulting equation, we may obtain
\begin{align}\label{173}
 &g(\phi Y,\phi Z)\triangle \alpha  + \rho g(\phi Y,\phi Z)\triangle R - 2g(\phi Z,\nabla _{\phi Y} \vec{\nabla} \alpha ) - 2\rho g(Z,\nabla _{\phi Y} \vec{\nabla} R) + 2\xi (\beta )g(\phi Y,Z) \\
&+ [V(\alpha ) - 2\alpha (\lambda  + \alpha  + \rho R) + 4\beta ]g(\phi Y,\phi Z) - 2\beta g(h'Y,Z) + 2\xi (\beta )g(hY,Z) = 0.\nonumber
\end{align}
The above formula holds for any $Y,Z \in \mathcal{X}(M)$, so interchanging $Y$ and $Z$ of relation (\ref{173}) yields a new equation, subtracting the resulting equation from (\ref{173}) and applying the relation $g(\nabla_X \vec{\nabla} f ,Y)=g(\nabla_Y \vec{\nabla} f ,X)$ again we may obtain $\xi (\beta )g(\phi Y,Z) = 0$ for any vector fields $Y$ and $Z$ on $M$, then it follows that
\begin{equation}\label{183}
\xi (\beta)=0.
\end{equation}
Using above equality in relation \ref{113} we get $\vec{\nabla} \alpha  + \vec{\nabla} \beta  - 4\beta \xi  = 0$, taking the inner product of this relation with $\xi$ we obtain $\xi (\beta)=4\beta $. Recall that $M$ is an $\eta$-Einstein almost
Kenmotsu manifold of dimension 3 if and only if $\xi$ belongs to the generalized $k$-nullity
distribution with $k=(\alpha +\beta)/2$, then by applying Lemma \ref{lem}, we obtain $\xi (k) = -4(k +1)$,
making use of $k=(\alpha +\beta)/2$ , $\xi (\beta ) = 0$ and $\xi (\beta)=4\beta $ in this relation we obtain
\begin{equation}\label{193}
\alpha +2\beta +2=0.
\end{equation}
It is easy to obtain from (\ref{183}) and (\ref{193}) that $\beta=0$ and $\alpha =-2$. However, in fact, in this context we have $k=(\alpha +\beta)/2 =-1$, which contradicts the assumption $k<-1$ everywhere on $M$. Thus the Case 1 never happens.\\
{\bf Case 2:} In the case where $k=-1$, from \ref{eq} we get that $h=0$ and hence $M$ is a Kenmotsu manifold (see Proposition 3 of \cite{9}). Also, according to Lemma 1 of (\cite{10}) the Ricci curvature tensor of $(M,g)$ can be written as follows.
\begin{equation}\label{203}
\mathrm{Ric}(X,Y) = (1 + \frac{R}
{2})g(X,Y) - (3 + \frac{R}
{2})\eta (X)\eta (Y)\,
\end{equation}
which means $\alpha =1+R/2$ and $\beta=-(3+R/2)$. By replacing these equalities in \ref{113}, we may obtain
\begin{equation}\label{213}
\xi (R) + 2(6 + R) = 0.
\end{equation}
Also, relation (\ref{83}) can be rewritten as follows
\begin{align}
 2(\mathcal{L}_V \nabla )(X,Y) &= (1 + 2\rho )[g(X,Y)\vec{\nabla} R - X(R)Y - Y(R)X] + X(R)\eta (Y)\xi \\
 &+ Y(R)\eta (X)\xi - \eta (X)\eta (Y)\vec{\nabla} R + 2(6 + R)[g(X,Y)\xi  - \eta (X)\eta (Y)\xi ]. \nonumber 
\end{align}

Hence, we can write
\begin{equation}\label{223}
2(\mathcal{L}_V \nabla )(Y,\xi ) = \xi (R)[\phi ^2  - 2\rho Y] - 2\rho Y(R)\xi .
\end{equation}
By differentiation of (\ref{223}) along an arbitrary vector field $X$, we get
\begin{align}
  &2(\nabla _X \mathcal{L}_V \nabla )(Y,\xi ) + 2(\mathcal{L}_V \nabla )(Y,X)  \\
  =&X(\xi (R))\phi ^2 (Y) + \xi (R)[g(X,Y)\xi  + \eta (Y)X - \eta (X)Y - \eta (X)\eta (Y)\xi ] \nonumber\\
 &+ 2\rho [g(X,Y)\vec{\nabla} R - X(R)Y - Y(R)X] - 2\rho [(\nabla _X dR)(\xi )Y + (\nabla _X dR)(Y)\xi ]. \nonumber  
\end{align}
With the help of the above formula and (\ref{133}) we can write
\begin{equation}\label{233}
2(\mathcal{L}_V \nabla )(X,Y)Z = X(\xi (R))\phi ^2 (Y) - Y(\xi (R))\phi ^2 (X) + 2\xi (R)[\eta (Y)X - \eta (X)Y].\,
\end{equation}
On the other hand, the equality $R(X,Y)\xi  = \eta (X)Y - \eta (Y)\xi$ holds in any Kenmotsu manifolds and by differentiation both sides of this equality along the vector field $V$ and making use of (\ref{43}) we obtain
\begin{equation}\label{243}
(\mathcal{L}_V R)(X,Y)\xi  + R(X,Y)\mathcal{L}_V \xi  = (4 - 2\lambda )[\eta (X)Y - \eta (Y)X]\,\, + g(X,\mathcal{L}_V \xi )Y - g(Y,\mathcal{L}_V \xi )X.
\end{equation}
Comparing (\ref{233}) and (\ref{243}) yields an equation and then contracting the result equation over $X$ and making use of (\ref{203}) again, we get
\begin{align}
&(6 + R)g(Y,\mathcal{L}_V \xi ) - (6 + R)\eta (Y)\eta (\mathcal{L}_V \xi )=\\
&- Y(\xi (R)) - (\xi (\xi (R)))\eta (Y) - 4(4 - 2\lambda  + \xi (R))\eta (Y).\nonumber
\end{align}
If we set $Y=\xi$ in the above formula then, by (\ref{213}) we get $\lambda =2$ which shows the soliton is expanding. Now, by Theorem 1 of \cite{9} we have completed the proof.
\end{proof}
%%%%%%%%%%%%%%%%%%%%%%%%%%%%%%%%%%
\section{Example}
In what follows we consider $M=\mathbb{R}\times_\gamma N$, where, $N$ is a Riemannian surface with constant negative sectional curvature (a Kahler manifold), $\mathbb{R}$ is real line and $\gamma=\gamma (t)$ is warp function. In fact, we consider the following warped  metric on $M$
\[
g=\dfrac{h}{\gamma^2(t)}+dt^2,
\]
where, $h$ is a Riemannian metric with constant curvature. So, $M$ is a $\beta$-Kenmotsu manifold with $\beta={\gamma'(t)/ \gamma(t)}$ (see \cite{1}). Suppose that $R$ stands for scalar curvature of $M$ then, an argument analogous to that of example 2.10 in \cite{3} shows that $g$ is a $\rho$-soliton with vector field $V=-\mu+f\dfrac{\partial}{\partial t}$ if and only if 
\[(\ln \gamma)''-\gamma^2 K^N=0,\]
where, $K^N$ denotes the Gaussian curvature of $N$ and,
\[f=\dfrac{\gamma'' +(\lambda +\rho R)\gamma +K^N \gamma^3}{{\gamma '}^2}-\dfrac{3\gamma '}{\gamma}.\]
If we just restrict attention to the case in which $K^N=-1$, then this leads us to the following ordinary differential equation,
\[(\ln \gamma)''+\gamma^2 =0.\]
The curve $\gamma =\dfrac{1}{\cosh t}$ is a particular solution for the above equation and for which we have
\[g=(\cosh t)^2 h+dt^2.\]
Hence, $M$ is a $\beta$-Kenmotsu manifold with $\beta =\tanh t$ (see \cite{1}). By a $D$-homothetic transformation we derive a Kenmotsu metric on $M$. Let
\[g^*=\sigma g+(1-\sigma )\eta \otimes \eta,\]
where, $\sigma$ is a positive function which depends only on $\xi =\dfrac{\partial}{\partial t}$. Using Lemma 4.1 in the paper \cite{2}, first we derive a $\beta$-Kenmotsu manifold $(M^*,\phi,\xi ,\eta ,g^*)$ with $\beta^* =\beta+\dfrac{\xi (\sigma)}{2\sigma}$. Now, we wish to choose $\beta^*$ such that the smooth manifold $M^*$ be a Kenmotsu manifold. It is sufficient to set  $1=\beta+\dfrac{\xi (\sigma)}{2\sigma}$, which leads us to
\[\dfrac{\partial}{\partial t}(\ln \sigma)=2(1-\beta).\]
The curve $\sigma=\dfrac{e^{2t}}{(\cosh t)^2}$ satisfies the above equation and so, the metric $g^*=e^{2t}h +dt^2$ is the desired Kenmotsu metric.
%%%%%%%%%%%%%%%%%%%%%%%%%%%%%%%%%%%%%%%%%%
\section{Conclusion}
In this paper, we showed that if the metric of a three dimensional almost Kenmotsu manifold, be a $\rho$-soliton then, the underlying manifold is a Kensmotsu manifold with constant sectional curvature $-1$ and the soliton is expanding. Of course, we have considered 3-dimensional manifolds and extending the results of this paper to higher dimensional spaces will be a good project.
%%%%%%%%%%%%%%%%%%%%%%%%%%%%%%%%%%%


\begin{thebibliography}{99}
\bibitem{1}
Alger P., Blair DE., Carriazo A., Generalized Sasakian-space-forms, Israel J. Math., 141 (2004), 83-157.
\bibitem{2}
 Alegre P., Carriazo A.,  Generalized Sasakian space forms and conformal changes of the metric,  Results Math, 59 (2011), 485-493.
\bibitem{3}
 Baird P., Danielo L., Three-dimensional Ricci solitons which project to surfaces. J Reine Angew Math 608 (2007), 65-91.
\bibitem{4}
 Blair D. E., Riemannian Geometry of Contact and Symplectic Manifolds, Progress in Mathematics, Vol. 203, Birkhauser, 2010.
\bibitem{5}
 Bourguignon J. P., Ricci curvature and Einstein metrics, Global differential geometry and global analysis, Lecture nots in Math., 838 (1981), 42-63.
\bibitem{6}
 Catino G., Cremaschi L., Djadli Z., Mantegazza C., Mazzieri L., The Ricci-Bourguignon flow, Pacific J. Math., 2015.
\bibitem{7}
 Cho J. T., Almost contact 3-manifolds and Ricci solitons, Int. J. Geom. Methods Mod. Phys., 10 (1) (2013), 1220022 (7 pages).
\bibitem{8}
 De U. C., Turan M., Yildiz A., De A., Ricci solitons and gradient Ricci solitons on 3-dimensional normal almost contact metric manifolds, Publ. Math. Debrecen, 80 (2012), 127-142.
\bibitem{9}
 Dileo G., Pastore A. M., Almost Kenmotsu manifolds and local symmetry, Bull. Belg. Math. Soc. Simon Stevin, 14 (2007), 343-354.
\bibitem{10}
 Ghosh  A., Kenmotsu 3-metric as a Ricci soliton, Chaos Solitons Fractals, 44 (2011), 647-650.
\bibitem{11}
 Kenmotsu K., A class of almost contact Riemannian manifolds, Tohoku Math. J.,  24 (1972), 93-103.
\bibitem{12}
 Pastore A. M., Saltarelli V., Generalized nullity distributions on almost Kenmotsu manifolds, Int. Electron. J. Geom., 4 (2) (2011), 168-183.
\bibitem{13}
 Saltarelli V., Three-dimensional almost Kenmotsu manifolds satisfying certain nullity distributions, Bull. Malays. Math. Sci. Soc., 38 (2015), 437-459.
\bibitem{14}
Turan M., De U. C.,  Yildiz A., Ricci solitons and gradient Ricci solitons in three dimensional trans-Sasakian manifolds, Filomat, 26 (2) (2012), 363-370.
\bibitem{15}
 Wang Y., Liu X.,  Ricci solitons on three-dimensional  $\eta$-Einstein almost  Kenmotsu manifolds, Taiwanese Journal of Mathematics, V.19, N. 1, (2015), 91-100.
\bibitem{16}
 Yano K., Integral Formulas in Riemannian Geometry, Marcel Dekker, New York, 1970.
 %%%%%%%%%%%%%%%%%%%%%%%%%%%%%%%%%%%%%%%%%%%%%%%%%%%%%%%%%%%%%%%%%%%%%%%%%%%%%%%%%%%%%%%%%%%%%%%%%%%%%%%%%%%%%%%%%%%%%%%%%
\end{thebibliography}
\end{document}